\documentclass[]{amsart}
\usepackage{amsmath, amssymb}
\newfont {\cyr} {wncyr10}
\renewcommand{\labelenumi}{{(\roman{enumi})}}

\usepackage{amsfonts}
\usepackage{amsmath}
\usepackage{amssymb}
\usepackage{setspace}
\usepackage{a4}

\usepackage{color}

\newtheorem{lemma}{Lemma}[section]
\newtheorem{theorem}[lemma]{Theorem}

\newcounter{claim}[lemma]

\begin{document}

\def \b {\beta}
\font\cyrl=wncyss10
\newcommand{\m}{$\,\textrm{max}\,$}
\newcommand{\w}{\widehat}
\newcommand{\wi}{\widehat}
\newcommand{\ov}{\overline}
\newcommand{\N}{\mathbb{N}}
\def \P{\mathbb{P}}
\def \Hom{\mathrm {Hom}}
\newcommand{\E}{\mathcal{E}}
\newcommand{\MM}{\mathcal{M}}
\newcommand{\K}{\mathcal{K}}
\newcommand{\wt}{\widetilde}
\newcommand{\wh}{\widehat}
\newcommand{\ti}{\tilde}

\def \Z  {\mathbb Z}
\newcommand{\ch}{$\,\textrm{char}$}
\newcommand{\sy}{$\,\textrm{Syl}$}
\newcommand{\sym}{\mathcal{S}}
\newcommand{\au}{$\,\textrm{Aut}$}
\newcommand{\out}{$\,\textrm{Out}$}
\newcommand{\rank}{$\,\textrm{r}$}

\def \<{\langle }
\def \>{\rangle }
\def \L{\mathcal{L}}

\def \Sym {\mathrm {Sym}}
\def \Alt {\mathrm {Alt}}
\def \Mat {\mathrm {Mat}}
\def \PGL {\mathrm {PGL}}
\def \PSL {\mathrm {PSL}}
\def \PSU {\mathrm {PSU}}
\def \PSp {\mathrm {PSp}}
\def \Sp {\mathrm {Sp}}
\def \SL {\mathrm {SL}}
\def \GL {\mathrm {GL}}
\def \GU {\mathrm {GU}}
\def \B {\mathrm {B}}
\def \PSO {\mathrm {PSO}}
\def \GF {\mathrm {GF}}
\def \M {\mathrm {M}}
\Large{\textbf{}}
\def \l {\lambda}
\def \diag {\mathrm {diag}}
\def \Aut {\mathrm {Aut}}

\renewcommand{\labelenumi}{(\roman{enumi})}

\title  {A note on groups in which the centraliser of every element of order $5$ is a $5$-group}
 \author{Sarah Astill}
 \author{Chris Parker}
  \author{Rebecca Waldecker}

\address{Sarah Astill\\
Department of Mathematics\\ The University of Bristol\\\ University Walk\\ Bristol BS8 1TW \\United
Kingdom}\email{sarah.astill@bristol.ac.uk}
\address{Chris Parker\\
School of Mathematics\\
University of Birmingham\\
Edgbaston\\
Birmingham B15 2TT\\
United Kingdom} \email{c.w.parker@bham.ac.uk}

\address{Rebecca Waldecker\\
Institut f\"ur Mathematik\\ Universit\"at Halle--Wittenberg\\
Theodor--Lieser--Str. 5\\ 06120 Halle\\ Germany}
\email{rebecca.waldecker@mathematik.uni-halle.de}

\maketitle

\section{introduction}

\medskip

In \cite{c55},  Dolfi, Jabara  and Lucido determine the finite groups in which  the centraliser of every element
of order $5$ is a $5$-group. However, the main result in \cite {c55} is flawed: when looking at
groups $G $ with $F^*(G/F(G)) \cong \Alt(5)$ and $F(G)$ of odd order the authors of \cite{c55} erroneously prove that $F(G)$ must be abelian. We
believe that the error occurs on page 1060 of their paper. Our main result is a corrected version of their theorem:

\begin{theorem}[Dolfi, Jabara  and Lucido]\label{MainThm}  Suppose that $G$ is a finite group in which the centraliser of every element of order $5$ is a $5$-group. Then one of the following holds:
\begin{enumerate}
\item $G$ is a $5$-group.
\item $G$ is a soluble group and one of the following hold:
\begin{enumerate}
\item $ G$ is a soluble Frobenius group such that either the Frobenius kernel or a Frobenius complement
is a 5-group;
\item   $G$ is a 2-Frobenius group such that $F(G)$ is a $5'$-group and $G/ F(G)$ is a Frobenius group,
whose kernel is a cyclic $5$-group and whose complement is cyclic of  order 2 or 4; or
\item  $G$ is a 2-Frobenius group such that $F(G)$ is a 5-group and $G/ F(G)$ is a Frobenius group,
whose kernel is a cyclic $5'$-group and whose complement is a cyclic 5-group.
\end{enumerate}
\item $G/F(G) \cong \Alt(5)$ or $\Sym(5)$ and $F(G)= O_2(G)O_{2'}(G)$  where $O_2(G)$ is nilpotent of
class at most three and $O_{2'}(G)=O_{\{2,5\}'}(G)$ is  nilpotent of class at most two.
\item $G/F(G) \cong \Alt(6)$,  $\Sym(6)$  or $\M(9)$ and $F(G)= O_2(G)O_3(G)$ where  $O_2(G)$ and $O_3(G)$
are elementary abelian.
\item $G/F(G) \cong \Alt(7)$ and $F(G)=O_2(G)$ is elementary abelian.
\item $G/F(G) \cong {}^2\mathrm B_2(8)$  or ${}^2\B_2(32)$ and $F(G)=O_2(G)$ is elementary abelian.
\item $G/F(G) \cong \PSU_4(2)$ or $\Aut(\PSU_4(2))$ and $F(G)=O_2(G)$ is elementary abelian.
\item $G/F(G) \cong \PSL_2(49)$,  $\PGL_2(49) $ or $\M(49)$ and $F(G)=O_7(G)$ is elementary abelian.
\item $E(G) \cong \PSL_2(5^e)$ with $e \ge 2$ and either $G \cong \PSL_2(5^e)$ or $G \cong \PGL_2(5^e)$ or
$e$ is even and $G\cong \M(5^e)$.
    \item $E(G) \cong \PSL_2(p)$ where $p$ is a prime which can be written as $p = 2\cdot5^e \pm 1$  for some non-negative integer $e$.
\item  $G \cong \PSL_2(11)$,  $\PSL_3(4)$, $\PSL_3(4){:}2_f$, $\PSL_3(4){:}2_i$, $\PSp_4(7)$, $\PSU_4(3)$,  $\Mat(11)$ or $\Mat(22)$.
\end{enumerate}
\end{theorem}
We    adopt
notation from \cite{c55}, so, for all odd primes $p$ and all $m \in \N$, we write $\M(p^{2m})$ to denote
the non-split extension of $\PSL_2(p^{2m})$ such that $|\M(p^{2m}):\PSL_2(p^{2m})|=2$.
In part (xi) of Theorem~\ref{MainThm}, we denote by $2_f$ a field automorphism of order $2$ and by $2_i$ the inverse transpose automorphism.  Throughout the statement, elementary abelian groups may be trivial. We finally point out  that the $G/F(G)$-modules involved in $F(G)$ are explicitly known in all cases.

The paper is organised as follows. In Section 2,  we construct examples of $r$-groups of class two for all
odd primes $r$, $r \neq 5$, which admit an action of $\Alt(5)$ with an element of order $5$ acting without non-trivial fixed
points. This demonstrates that the statement in \cite{c55} is false.  However,  in Section 3, we show
that there are no groups $G$ with $O_{\{2,5\}'}(G)$ nilpotent of class three with
$G/O_{\{2,5\}'}(G) \cong \Alt(5)$ and with
an element of order $5$ acting fixed point
freely. This is the main contribution of this paper and our proof of this fact is guided by some of the arguments in \cite{Holt}.

We further note that the authors of \cite{c55} use the fact that if $\Alt(5)$ acts on a $7$-group,  then the
$7$-group is abelian to show that $\PSL_2(49)$ cannot act on a non-abelian $7$-group.    In Section 4, we provide an alternative
proof for this fact and prove, in addition, that such an abelian group must  be elementary abelian. We also show that, in part (iv) of Theorem~\ref{MainThm}, the subgroup $O_3(G)$ is elementary abelian. This sharpens the result
stated in \cite{c55}.
\bigskip

\noindent {\bf Acknowledgement.}  The second author is  grateful to the DFG for their support and both the first and second author thank the mathematics department in Halle for their hospitality.

\bigskip

\section{Construction of a class two group admitting $\Alt(5)$ with an element of order 5 acting fixed point freely}

\medskip
In this section we demonstrate that Theorem \ref{MainThm}~(iii) cannot be strengthened to say that $O_{2'}(G)$ is abelian by constructing examples with $O_{2'}(G)$ of class two.

Suppose that $r$ is an odd prime with $r \neq 5$ and that $W$ is a
$5$-dimensional vector space over $\GF(r)$ with basis $\{a_1,\dots, a_5\}$. Let $X:=\Sym(5)$ naturally  permute this basis. This permutation action gives rise to a faithful linear action of $X$ on $W$ which is in fact a reduction modulo $r$ of the corresponding integral representation of $X$. We
let $V$ be the $\GF(r) X$-submodule of $W$ with basis $\mathfrak V=
\{v_1, \dots, v_4\}$ where, for all $1\le i \le 4$, we set $v_i:= a_1-a_{i+1}$. It is easy to check that an element of order $5$ in $X$ acts fixed point freely on $V$ and that $W=V \oplus \langle a_1+a_2+a_3+a_4+a_5\rangle$.
Finally we let $Y:= X' \cong \Alt(5)$ and note that the restriction of $V$ to $Y$  is also an irreducible $\GF(r)Y$-module. We also denote this module by $V$.
\medskip

Recall that $V \wedge V$ is isomorphic to the submodule of $V \otimes V$ generated by the vectors  $ v_i \otimes
v_j - v_j \otimes v_i$ for $1\le i < j \le 4$.

\begin{lemma} \label{modfacts} The $\GF(r)Y$-module $V \wedge V$ has no composition factor isomorphic to $V$.
\end{lemma}

\begin{proof}
Let $U:=V \wedge V$. Calculating most easily with  $S = \langle (2,3)(4,5),(2,4)(3,5)\rangle$, we obtain that $C_U(S)=0$. Since $C_V(S) = \langle v_1+v_2+v_3+v_4\rangle$, the result follows.
\end{proof}

We need the following statement about self-extensions of $V$ by $V$.

\begin{lemma}\label{coho} Any $\GF(r)Y$-module with all composition
factors isomorphic to $V$ is completely reducible.
\end{lemma}

\begin{proof} We need to show that $\mathrm {Ext}^1_Y(V,V)= 0$.
(For notation see for example \cite{benson}.)
If $r > 3$, then, as $r \neq 5$, the result follows
from Maschke's Theorem. So suppose that $r=3$. Let $H$ denote a subgroup of $Y$ which is isomorphic to $\Alt(4)$. We recall that $V$ is a direct summand of the natural
permutation module $W$ for $Y$. Hence if we let $1_H$ denote the trivial $\GF(3)H$-module, then  $W$ is the induced module $1_H\uparrow ^Y$.
It follows that
$$\mathrm {Ext}_Y^1(W,V)= \mathrm {Ext}_Y^1(1_H\uparrow ^Y,V)= \mathrm
{Ext}_H^1(1_H,V\downarrow_ H),$$
where the second equality comes from Shapiro's Lemma \cite[Corollary 3.3.2]{benson}.
Now $V\downarrow_ H$ is a direct sum of a faithful $3$-dimensional module and a trivial module. As $H$ contains a Klein fours subgroup and as this subgroup acts coprimely on the
$3$-dimensional module, this module only has trivial extensions with the trivial $H$-module. 

Since $\dim \mathrm {Ext}^1_Y(1_H,1_H)= 1$ (see for example \cite[Corollary 3.5.2]{benson}), we have that $\dim \mathrm {Ext}^1_Y(1_H,V\downarrow H)= 1$.
On the other hand 
$$\mathrm {Ext}^1_Y(W,V)= \mathrm
{Ext}^1_Y(1_Y\oplus V,V)= \mathrm {Ext}^1_Y(1_Y,V)\oplus \mathrm {Ext}^1_Y(V,V).$$ Now let $Y \ge D \cong \mathrm
{Dih}(10)$. Then $1_D\uparrow^Y$ is a uniserial module with socle and head of dimension $1$ and heart of dimension
$4$. Hence $\dim \mathrm {Ext}^1_Y(1_Y,V) \ge 1$. We  infer that $\mathrm {Ext}^1_Y(V,V)=0$ as claimed.
\end{proof}

Constructions of $r$-groups admitting the action of a further group are intimately related to tensor products and homomorphisms between modules. We therefore study $\Hom(V,V)$ in the next lemma.

Let  $\sigma = v_1+v_2+v_3+v_4$ and define $\theta \in \Hom(V,V)$ as follows: for all $1\le i \le 4$,
$$v_i \mapsto
v_i\theta :=\sigma - v_i.$$

\medskip
With respect to the basis $\mathfrak V$,
we calculate that $\theta$ has matrix $$\left(\begin{array}{cccc}
0&1&1&1\\1&0&1&1\\1&1&0&1\\1&1&1&0\end{array}\right).$$

\medskip
\begin{lemma}\label{homsubmod} The $\GF(r)X$-submodule of $\Hom(V,V)$
generated by $\theta$ has dimension $4$ and is isomorphic to $V$.
\end{lemma}

\begin{proof} Let $X_1 = \mathrm {Stab}_X(1) \cong \Sym(4)$. Then, for all $1\le i \le
4$ and for all $\pi \in X_1$, we have $$v_i\pi^{-1} \theta\pi=
(v_{(i+1)\pi^{-1}-1})\theta\pi = (\sigma - v_{((i+1)\pi^{-1}-1)\pi}) =
\sigma - v_i= v_i\theta.$$ Since $v_1(1,2)\theta
(1,2)\neq v_1\theta$, the orbit of $X$ containing $\theta$ has
exactly $5$ elements. Therefore the subspace $U$ of $\Hom(V,V)$
spanned by $\theta$ has dimension $4$ or $5$. If $\dim U = 5$, then
the sum $\tau$ of the translates of $\theta$ under $X$ is
centralised by $X$ and is non-trivial. By Schur's Lemma we then have
that $\tau$ is a scalar matrix. Since $r$ is odd, if $\tau$ is non-zero, it has
non-zero trace. However, $\theta$ and its translates have trace $0$
and therefore $\tau$ must also have trace $0$. Thus $\tau=0$ and
$\dim U=4$.
\end{proof}

\begin{theorem}\label{cls2}
For all odd primes $r$, $r \neq 5$, there are $r$-groups of class two
which admit $\Sym(5)$ with an element of order $5$ acting without
fixed points.
\end{theorem}

\begin{proof} Let $U$ be an $8$-dimensional vector space over
$\GF(r)$. We define a subgroup $J$ of $\GL_8(r)$ by
$$J=\langle \left(
\begin{array}{cc}I_4&0\\\theta&I_4\end{array}\right),\left(
\begin{array}{cc}x_\pi&0\\0&x_\pi\end{array}\right)\mid \pi \in
X\rangle$$ where $x_\pi$ is the matrix corresponding to the action of $\pi\in X$ on $V$ with respect to the basis
$\mathfrak V$ and $I_4$ is the $4\times 4$ identity matrix. Then $J \cong r^4{:}\Sym(5)$ by Lemma~\ref{homsubmod}. Set $K= U \rtimes J$. Then $O_r(K)$ has
order $r^{12}$ and class two. Furthermore, the elements of order $5$ in $K$ are self-centralising.
\end{proof}

Table~\ref{tab1} describes  an element  $\gamma\in \Hom(V\otimes V,V)$ by defining the images of the basic tensors.
\begin{table}[h]
$$\begin{array}{c|cccc}
\otimes&v_1&v_2&v_3&v_4\\
\hline
v_1&-5v_1+2\sigma&\sigma&\sigma&\sigma\\
v_2&\sigma&-5v_2+2\sigma&\sigma&\sigma\\
v_3&\sigma&\sigma&-5v_3+2\sigma&\sigma\\
v_4&\sigma&\sigma&\sigma&-5v_4+2\sigma\\
\end{array}.
$$
\caption{The module homomorphism $\gamma$}\label{tab1}
\end{table}

The next lemma will be used in Section 3.

\begin{lemma}\label{unique} Up to scalar multiplication $\gamma$ is the unique element of $\Hom_{Y}(V\otimes V,V)= \Hom_X(V\otimes V,V)$.
\end{lemma}

\begin{proof} We note that $X=\langle (1,2), (2,3,4,5)\rangle$ and that the second generator permutes
$\{v_1, v_2,v_3,v_4\}$ as a $4$-cycle. In particular, $\gamma$ commutes with this element. Hence we only need  to
verify that $\gamma$ commutes with the transposition. As an illustration we show that $(v_1\otimes v_2)\gamma
(1,2)= (v_1\otimes v_2)(1,2)\gamma$. The left hand side is seen to be $\sigma(1,2)=  -5 v_1+\sigma$, while the
right hand side equals $-(v_1\otimes v_2)\gamma + (v_1\otimes v_1)\gamma= \sigma -5v_1$. Thus $\gamma \in \Hom_X(V\otimes V,V)$.

Assume that $\mu \in \Hom_{Y}(V\otimes V,V)$ is not a scalar multiple of $\gamma$. We consider $(v_1\otimes v_1)\mu = \sum_{i=1}^5\lambda_ia_i$ where
$\l_i \in \GF(r)$ and $\sum_{i=1}^5\lambda_i= 0$. Using the fact that $\mu$ commutes with elements from $Y$, we
see, by applying  $(3,4,5)$, that $\lambda_3=\lambda_4=\lambda_5$ which we define to be $\lambda$. Next, applying
$(1,2)(3,4)$, we deduce that $\lambda_1=\lambda_2$. So $$(v_1\otimes v_1)\mu = \l_1(a_1+a_2)+\l(a_3+a_4+a_5)$$ and,
therefore, $2\lambda_1 + 3\l =0$. Since $\mu$ is not the zero map, after adjusting $\mu$ by a  scalar, we may suppose that  $$(v_1\otimes v_1)\mu  = -5v_1+2\sigma= (v_1\otimes
v_1)\gamma.$$ Replacing $\mu $ with $\mu - \gamma$, we only need to consider the case where $\mu$ maps
$v_1\otimes v_1$ to zero. The action of $Y$ now gives  $(v_j\otimes v_j)\mu =0$ for all $1\le j \le 4$.
  Let
$(1,j,k)$ be a $3$-cycle in $Y$. Since $(v_j\otimes v_j)\mu =0$, we have that $$0= (v_j\otimes v_j)(1,j,k)\mu =
((v_k-v_j)\otimes (v_k-v_j))\mu = -(v_k\otimes v_j)\mu- (v_j\otimes v_k)\mu.$$ It follows that $\mu$ is
alternating and hence that $V \wedge V$ has a quotient isomorphic to $V$. But then Lemma~\ref{modfacts} implies that $\mu$ is the zero map and this is our contradiction. The lemma is now proved.
\end{proof}

\bigskip

\section{A non-existence Theorem about class three groups admitting $\Alt(5)$}

\medskip
We have seen in the previous section that $\Alt(5)$ can act on a class two group with an element of order 5 acting fixed point freely. Our objective in this section is to show that if $E(G/F(G)) \cong \Alt(5)$, then Theorem~\ref{MainThm} (iii) holds. In particular, we show that $\Alt(5)$ cannot act on a class three group of odd order with an element of order 5 acting fixed point freely.
Hence suppose that $G$ is a finite group such that the centraliser of every element of order $5$ is a $5$-group and $E(G/F(G)) \cong \Alt(5)$. Assume that $O_5(G)\neq 1$. Then, as $F(G)$ is nilpotent, the fact that elements of order $5$ are $5$-groups  implies that $F(G)= O_5(G)$. Let $B$ be a Sylow $2$-subgroup of $G$. Then $B$ is elementary abelian of order $4$. By coprime action we have that $O_5(G) = \langle C_{O_5(G)}(b)\mid b \in B^\#\rangle$. So, as the elements of order $5$ in $G$ do not commute with involutions, we have a contradiction. Thus $O_5(G)=1$. Therefore  we only have to restrict the
structure of $O_r(G)$ for all primes $r\neq 5$ dividing $|F(G)|$. If  $r=2$, then \cite[Lemma~4.1]{prince} implies that $G$ contains a subgroup isomorphic to $\Alt(5)$ and so  we have that $O_2(G)$ admits
$\Alt(5)$ with an element of order $5$ acting fixed point freely. In this case  \cite[Theorem
2]{Holt} yields that $O_2(G)$ has class at most three. So we may assume that $r$ is odd with $r \neq 5$. Our aim
will be achieved once we prove the following theorem:

\begin{theorem}\label{cl3} Suppose that $r$ is an odd prime with $r\neq 5$ and that $G$ is a group with $R:=F(G)=O_r(G)$ an $r$-group and $F^*(G/R)\cong \Alt(5)$. If
an element of order $5$ in $G$ acts fixed point freely on $R$, then $R$ has class at most two.
\end{theorem}

The remainder of this section is devoted to proving Theorem~\ref{cl3}, and
it suffices to consider the case where  $G/O_r(G) \cong \Alt(5)$ and $O_r(G)$ has class three.

The most direct approach to the proof of Theorem~\ref{cl3} would choose $G$ of minimal order with $R:=O_r(G)$ of class three and then derive a contradiction. This approach works very smoothly when $r \neq 3$ because in these cases $G$ splits over $R$. However, when $r=3$  it is possible that $G$ does not split over $R$ (as there exist groups $X$ with $X/O_3(X) \cong \Alt(5)$ and $O_3(X)$ isomorphic to $V$ as an $X/O_3(X)$-module which do not contain subgroups isomorphic to $\Alt(5)$). Thus in the straightforward approach to the proof of Theorem~\ref{cl3}, when we choose a normal subgroup $U$ of $G$ such that $U$ is proper in $R$, there may not exist a group $G^*$ with $O_3(G^*)= U$. Consequently our inductive hypothesis is not strong enough to assert that $U$ has class at most two.

Instead, we assume that Theorem~\ref{cl3} is false and we choose a counter-example $G$ of minimal order such that $R:=O_r(G)$ has class three. Then we choose a normal subgroup $Q$ of $G$ that is contained in $R$ and that is minimal with respect to having class three. We set $\ov{G}:=G/R$ and fix all this notation.

\begin{lemma}\label{QProps}
The following hold:

\begin{enumerate}
\item $Q'$ is abelian.
\item Every normal subgroup of $G$ that is properly contained in $Q$ has class at most two.
\item $\Gamma_2(Q)$ is elementary abelian of order $r^4$ and is isomorphic to $V$ (as defined before Lemma \ref{unique}) as a $\GF(r)\ov{G}$-module.
\item $Q'/(Q' \cap Z(Q))$ is a $\GF(r)\ov{G}$-module.
\item $\Phi(Q) \le Z_2(Q) \le C_Q(Q') <Q$. In particular, $Q/Z_2(Q)$ is elementary abelian.

\item Every $G$-composition factor of $Q$ is isomorphic as a
$\GF(r)\ov{G}$-module to $V$. Every elementary abelian $G$-invariant
section of $Q$ which is centralised by $R$ is a direct sum of
modules isomorphic to $V$.
\end{enumerate}
\end{lemma}
\begin{proof} We have $Q' \le Z_2(Q)$ and $[Q',Z_2(Q)]=1 $ follows from the three subgroup lemma.
Thus $Q'\le Z_2(Q) \le C_Q(Q') <Q$ and, especially, (i) and some of the inclusions in (v) hold.

For part (ii) let $Q_0$ be a normal subgroup of $G$ that is properly contained in $Q$.
Then the minimal choice of $Q$ immediately gives that $Q_0$ has class at most two.

$\Gamma_2(Q)$ lies in $Z(Q)$ because $Q$ has class three, so it is abelian. Let $P:=\Phi(\Gamma_2(Q))$. Then $R$ centralises $P$ and hence, if $P \neq 1$, then $\wh{G}:=G/P$ is a counter-example of smaller order because $\wh{Q}$ is still of class three.
This contradiction shows that $P=1$ and so $\Gamma_2(Q)$ is elementary abelian and it follows similarly that $\Gamma_2(Q)$ is irreducible as a $\GF(r)\ov{G}$-module.
This proves (iii).

Let $a \in Q'$ and $q\in Q$. Then $[a^r,q]= [a,q]^r$. But $[a,q]\in
\Gamma_2(Q)$ which is elementary abelian by (iii). So $[a^r,q]=1$.
Hence $a^r \in Z(Q)$ and therefore $Q'/Q' \cap Z(Q)$ is elementary abelian.
For this factor group to be a $\GF(r)\ov{G}$-module
it remains to show that $[Q',R] \le Q' \cap Z(Q)$.

Of course $[Q',R] \le Q'$, and also
$[Q,Q',R] \le [\Gamma_2(Q),R]<\Gamma_2(Q)$. Therefore
$[Q,Q',R]=1$ by (iii).
As $R$ has class three, we know that
$[R,R'] \le Z(R)$ and hence $[Q,R] \le R' \le Z_2(R)$.
Therefore $[R,Q,Q']=1$ and hence the three subgroup lemma implies that $[Q',R,Q]=1$.
We deduce that $[Q',R] \le Z(Q)$ whence $[Q',R] \le Q' \cap Z(Q)$.

For the last statement of (v) suppose that $a, q_1, q_2 \in Q$. Then
$[a^r,q_1,q_2] = [a,q_1,q_2]^r =1$ which means that $a^r \in
Z_2(Q)$. Hence $Q/Z_2(Q)$ is elementary abelian.
As stated in \cite{c55}, $V$ is the unique $\GF(r)\ov{G}$-module which admits an element of
order $5$ from $\ov{G}$ acting fixed point freely. Thus Lemma~\ref{coho}
gives (vi).
\end{proof}

\begin{lemma}\label{3Ms}
Suppose that $M_1$, $M_2$ and $M_3$ are subgroups of $Q$ that are maximal subject to being normal in $G$ and contained in $Q$. Set $D= M_1\cap M_2\cap M_3$ and suppose that $|Q:D|=r^{12}$. Then $D \le Z_2(Q)$ and in particular $Q/Z_2(Q)$ is a direct product of at most three minimal
normal subgroups of $G/Z_2(Q)$ each of order $r^4$.
\end{lemma}

\begin{proof}
By Lemma~\ref{QProps} (ii), the subgroups $M_1$, $M_2$ and $M_3$ have class at most two. Thus for $1 \le i<j \le 3$, we have that
$[M_i\cap M_j ,M_j]\le M_j'\le  Z(M_j)$.
As $Q=M_iM_j$, we obtain that $$[M_i \cap M_j,Q] \le [M_i \cap M_j,M_i][M_i \cap M_j,M_j]\le
Z(M_i)Z(M_j).$$ In particular, $$[D,Q]\le Z(M_1)Z(M_2)\cap Z(M_2)Z(M_3)\cap Z(M_1)Z(M_3).$$
Since $|Q:D|=r^{12}$, we have $M_1 \cap M_2 \not \le M_3$. Therefore $Q= (M_1\cap M_2)M_3$ and
$$M_1= M_1 \cap (M_1\cap M_2)M_3= (M_1 \cap M_2)(M_1\cap M_3).$$
Similarly, $M_2= (M_1 \cap M_2)(M_2\cap M_3)$ and so $$Q= M_1M_2=(M_1 \cap M_2)(M_2\cap M_3)(M_1\cap M_3).$$
Hence it follows that $[D,Q,Q]$ is contained in
$$[Z(M_1)Z(M_2)\cap Z(M_2)Z(M_3)\cap Z(M_1)Z(M_3),(M_1 \cap M_2)(M_2\cap M_3)(M_1\cap
M_3)]=1$$ and consequently $D\le Z_2(Q)$.

We know from Lemma~\ref{QProps} (v) that $Q/Z_2(Q)$ is elementary abelian. As $R$ has class three and $Q \unlhd G$, we see that $[Q,R] \le R' \cap Q \le Z_2(R) \cap Q \le Z_2(Q)$, so
$Q/Z_2(Q)$ is centralised by $R$ and hence Lemma~\ref{QProps} (vi) implies that $Q/Z_2(Q)$ is a direct product of minimal normal subgroups of $G/Z_2(Q)$. If there are at least
three minimal normal subgroups of $G/Z_2(G)$ involved in this product, then there are exactly three by the first part of the lemma. This completes the proof.
\end{proof}

The remainder of the proof is organised as a series of claims.

\medskip
\begin{claim} \label{M1} Suppose that $M$ is a normal subgroup of $Q$ such that $M/Z_2(Q)$ is a minimal normal subgroup of $G/Z_2(Q)$. Then $M' \le Z(Q)$. In particular
$|Q/Z_2(Q)| > r^4$.
\end{claim}
\medskip

We know that $|M/Z_2(Q)| = r^4$ by hypothesis. Since $Z_2(Q)' \le Q'
\cap Z(Q)$ and $Q'/(Z(Q)\cap Q')$ is elementary abelian by Lemma~\ref{QProps} (iv), the commutator map defines a
$\GF(r)\ov G$-module homomorphism from $M/Z_2(Q) \otimes M/Z_2(Q)$ to $Q'/(Q'\cap Z(Q))$ which is well-defined as
$[Q,Z_2(Q)]\le Q'\cap Z(Q)$. Since $[a,b]=[b,a]^{-1}$ for all $a,b\in Q$, this map factors through $(M/Z_2(Q))
\wedge (M/Z_2(Q))$. But  $(M/Z_2(Q)) \wedge (M/Z_2(Q))$  has no $4$-dimensional quotients by
Lemma~\ref{modfacts}  and so the commutator map is trivial. Therefore, $M'\le Z(Q)\cap Q'$ as claimed. If
$|Q/Z_2(Q)| =r^4$, then we may take $M= Q$ and conclude that $Q$ has class two which is absurd.
\hfill$\blacksquare$

\bigskip

We now simultaneously define bases for all the $G$-composition
factors of $Q$ as follows.  Let $U$ be such a composition factor. Then  there is an isomorphism $\psi_U$
from $U$ to $V$. For $1\le i \le 4$, we let $u_i = (v_i)\psi_U^{-1}$.  We call this a \emph{standard basis} of $U$. Given a standard basis $u_1, \dots, u_4$, we define  $\sigma_u = u_1u_2u_3u_4$.
Recall the homomorphism  $\gamma$ from Table~\ref{tab1}. Suppose that $R,S$ and $T$ are $G$-composition factors with standard bases $r_1, \dots, r_4$ and $s_1, \dots , s_4$.  Then  define a map from $R\otimes S$ to $T$ by setting, for all $i,j \in \{1, \dots 4\}$:

$$r_i\otimes s_j \mapsto (r_i\psi_R\otimes s_j\psi_S)\gamma\psi_T^{-1}.$$
Now we can determine the image
of  $r_i\otimes s_j$ in $T$ represented in the standard basis for $T$  directly from the table describing
$\gamma$. Thus, for example, $r_1\otimes s_2 $ maps to $\sigma_t$ and $r_1\otimes  s_1$ maps to
$\sigma_t^2t_1^{-5}$. Replacing $\gamma$ by a scalar multiple $m\gamma$ we get that $r_1\otimes s_1$ maps to
$(\sigma_t^2t_1^{-5})^m$.

We are now going to exploit the Hall-Witt identity \cite[Lemma 5.6.1 (iv)]{Gorenstein} which in a class three
group such as $Q$ takes the form
$$[x,y,z][y,z,x][z,x,y]=1$$
for all $x,y,z \in Q$.

\medskip
\begin{claim}\label{M2} $|Q/Z_2(Q)| = r^{12}$.
\end{claim}
\medskip

As $|Q/Z_2(Q)| > r^4$ by (\ref{M1}), it is sufficient to exclude the case where $|Q/Z_2(Q)| = r^{8}$.
Hence assume that $|Q/Z_2(Q)| = r^{8}$.
We let $M_1$ and $M_2$ be normal subgroups of $G$ such that $Q= M_1M_2$ and such that $C:= M_1/Z_2(Q)$ and $D:=M_2/Z_2(Q)$ have
order $r^4$. We choose standard bases $c_1, \dots, c_4$ for $C$ and $d_1, \dots, d_4$ for $D$. Then, as in
\ref{M1}, the commutator map defines an $\GF(r)\ov G$-module homomorphism  from $C \otimes D$ to $Q'/(Q'\cap Z(Q))$. If
this map is trivial, then $[M_1,M_2]\le Q' \cap Z(Q)$ and  \ref{M1} implies that $Q'\le Z(Q)$ which is not
the case. By  Lemmas~\ref{unique} and \ref{QProps}(vi), the image of this commutator map is $4$-dimensional and
isomorphic to $V$. We let the image be $E$ and take a standard basis $e_1, \dots, e_4$ such that the commutator
map is defined by $\gamma$.
 Finally we may assume that $\Gamma_2(Q)=[M_1,M_2,M_2]$  and take a standard
basis $f_1, \dots, f_4$, again chosen so that $\gamma$ represents the commutator map from $E \otimes D$ to
$\Gamma_2(Q)$. We consider the Hall-Witt identity with the elements $c_1$, $d_1$ and $d_2$ (and here it is
critical to note that the identity is independent of the representatives for $c_1$, $d_1$ and $d_2$ that we
choose).  We have $[d_1,d_2] \in Z(Q)$ from \ref{M1} and so $[d_1,d_2,c_1]=1$. Calculating further commutators,
using the map $\gamma$ from Table~\ref{tab1}, we get
$$[c_1,d_1,d_2] =  [ (c_1\psi_C\otimes d_1\psi_D)\gamma \psi_E^{-1}, d_2]=[\sigma_e^2e_1^{-5} ,d_2] =
(\sigma_e^2e_1^{-5} \psi_E \otimes d_2\psi_D)\gamma\psi_F^{-1} = \sigma_f^{5}f_2^{-10}
$$ and, similarly,
$$[d_2,c_1,d_1] = [\sigma_e,d_1] = \sigma_f^5f_1^{-5}$$
which is a contradiction as their product is not the identity.

Thus $|Q/Z_2(Q)| = r^{12}$. \hfill$\blacksquare$

\bigskip

From \ref{M2} we have that $|Q/Z_2(Q)|=  r^{12}$. Let $M_1$, $M_2$ and $M_3$ be normal
subgroups of $G$ such that $Q=
M_1M_2M_3$ and such that $C:= M_1/Z_2(Q)$, $D:=M_2/Z_2(Q)$ and $E:=M_3/Z_2(Q)$ have order $r^4$. We take standard bases $c_1, \dots, c_4$, $d_1, \dots, d_4$ and $e_1, \dots, e_4$ for
$C$, $D$ and $E$ as described. Since $M_i' \le Z(Q)$ by \ref{M1}, we may assume that $F:= [M_1,M_2](Q' \cap
Z(Q))/(Q' \cap Z(Q))$ is non-trivial. 
We set 
$$H:= [M_1,M_3](Q' \cap Z(Q))/(Q' \cap Z(Q))$$ 
and 
$$J:= [M_2,M_3](Q'
\cap Z(Q))/(Q' \cap Z(Q)).$$ 
These may be trivial groups.  We let $f_1,\dots, f_4$ be a standard basis for $F$,
and when $H$ or $J$ is non-trivial, we take standard bases $h_1,\dots, h_4$ and $j_1,\dots,j_4$ for $H$ and $J$
respectively, where all the bases are chosen so that the commutator map is represented by $\gamma$. Let $k_1,
\dots, k_4$ be a standard basis for $\Gamma_2(Q)$. Notice that the basis $k_1,\dots, k_4$ cannot necessarily be
chosen so that all the possible commutator maps are represented by $\gamma$ itself. However, they are represented
by scalar multiples of $\gamma$ and the powers appearing in the images below do not affect the failure of the
Hall-Witt identity which we now proceed to check for the elements $c_1$, $d_1$ and $e_2$. Taking all commutator
maps to be $\gamma$, we calculate
$$[c_1,d_1,e_2]= [\sigma_f^2f_1^{-5}, e_2] = \sigma_k^5k_2^{-10},$$
$$[d_1,e_2,c_1] = \begin{cases}1&J=1\cr [\sigma_j,c_1] =
\sigma_k^5k_1^{-5}&\text{ otherwise}\end{cases}$$ and
$$[e_2,c_1,d_1] = \begin{cases}1&H=1\cr [\sigma_h,c_1] =
\sigma_k^5k_1^{-5}&\text{ otherwise}\end{cases}$$ and, in full generality, the images lie in the cyclic groups
generated by the elements presented above. Thus it follows that the Hall-Witt identity does not hold in the
putative group $Q$. This concludes the proof of the theorem.

\section{Final remarks}

In  this section we prove that if $E(G/F(G))\cong \PSL_2(49)$, then $F(G)$ is an elementary abelian $7$-group. We also take the opportunity to add more
detail to the statement of Theorem~\ref{MainThm} by showing that the abelian $3$-group in Theorem~\ref{MainThm}
(iv) is elementary abelian.

\begin{lemma}\label{notin}  The group $\PSL_2(49)$ is not isomorphic to a subgroup of
$\GL_4(\mathbb Z/49\mathbb Z)$.
\end{lemma}

\begin{proof} We calculate using {\sc Magma} \cite{Magma}. Let  $J \cong \PSL_2(49)$ and suppose that $J$ is a
subgroup of $K=\GL_4(\mathbb Z/49\mathbb Z)$. Then, as $J$ contains a subgroup isomorphic to $\Alt(5)$, we can again take the integral
version $A$ of $\Alt(5)$ from Section~{2} and suppose that it is a subgroup of $J$. Now we take a Sylow
$5$-subgroup $S$ of $A$ and determine $C_K(S)$ using {\sc Magma}. It has order $2^5\cdot 3\cdot 5^2\cdot 7^4$ and is
$5$-closed. We let $T \in \sy_5(C_K(S))$. Then $T$ is the unique subgroup of order $25$ in $K$ containing $S$. It
follows that $J = \langle A,T\rangle$ and a calculation shows that $J$ has order $2^4\cdot 3\cdot
5^2\cdot 7^8$. But this is a contradiction because $7^8$ does not divide the order of $J \cong \PSL_2(49)$. Hence $\PSL_2(49)$ is not isomorphic to a subgroup of $K$.

\end{proof}

\begin{lemma}\label{74 unique} Suppose that $k$ is an algebraically closed field of characteristic 7 and that $V$ is an irreducible $k\PSL_2(49)$-module which admits an
element of order $5$ without non-zero fixed points. Then $V$ is isomorphic to $N \otimes N^\sigma$ where $N$ is the
natural $k\SL_2(49)$-module and
$\sigma$ is the automorphism of $k$ obtained by raising every element to its seventh power. Furthermore, $V$ is not a composition factor of $V \otimes V$.
\end{lemma}

\begin{proof} By \cite[Section 30 (98)]{BN} we have  $V = U \otimes W^\sigma$ where $U$ and $W$ are basic
$k\SL_2(49)$-modules. The basic $k\SL_2(49)$-modules can be identified with the seven modules $U_j$, $0\le j \le
6$, obtained as degree $j$ homogeneous polynomials in $k[x,y]$.  Then $\dim U_j=j+1$.  Now let $\phi$ in $\SL_2(49) \le \GL_2(k)$ be of
order $5$. Then $\phi$ diagonalises in $\GL_2(k)$ and so we may assume that $\phi$ acts as the diagonal matrix
$\diag(\l,\l^{-1})$. It is straightforward to check that, for $j \ge 3$, $\phi$ has all possible non-trivial
eigenvalues on $U_j$ and so if $U_j$ or $U_j^\sigma$ appears in the tensor product defining $V$, then as $\phi$
acts fixed point freely on $V$,  one of the tensor factors in $V$ must be $U_0$. But then $\dim U_j$ must be even and we see
that $V$ is a module for $\SL_2(49)$, not for $\PSL_2(49)$. Similarly, we now deduce    the only contenders for
$V$ are $U_1 \otimes U_1^\sigma$, $U_2 \otimes U_1^\sigma$ and $U_1\otimes U_2^\sigma$. In the latter two cases
we again get a representation of $\SL_2(49)$ rather than $\PSL_2(49)$ and this shows that the only possibility is
that $V= U_1 \otimes U_1^\sigma$ and it is easy to check that this module has the required properties.

Finally, we have \begin{eqnarray*}V \otimes V&=& (U_1 \otimes U_1^\sigma)\otimes (U_1 \otimes U_1^\sigma)= (U_1
\otimes U_1)\otimes (U_1^\sigma \otimes U_1^\sigma) \\&=& (U_0 \oplus U_2)\otimes (U_0^\sigma\oplus U_2^\sigma)=
U_0 \oplus U_2^\sigma \oplus U_2 \oplus ( U_2\otimes U_2^\sigma)\end{eqnarray*} and none of these irreducible
summands  are isomorphic to $V$.
\end{proof}

\begin{lemma}\label{cr49} Any $\GF(7)\PSL_2(49)$-module which has all composition factors isomorphic to the module described in Lemma~\ref{74 unique} is completely reducible.

\end{lemma}
\begin{proof} A {\sc Magma} calculation has shown that $\mathrm
{Ext}^1_X(V,V)=\mathrm H^1(X,V\otimes V^*)= 0$. This proves the lemma. \end{proof}

\begin{lemma}\label{49}
Suppose that $r$ is a prime and that $G$ is a group such that $X:=G/O_r(G) \cong\PSL_2(49)$ and $G$ has an element of order $5$ acting fixed point freely on $O_r(G)$. Then $r=7$ and $O_7(G)$ is elementary abelian and completely reducible as a $\GF(7)X$-module.
\end{lemma}

\begin{proof}
Let $Q:=O_r(G)$.
By \cite{c55} we may suppose that $r=7$ and thus that $Q$ is a $7$-group.

Suppose first that $Q$ is abelian, but not elementary abelian, and  that $G$ has minimal order with these properties. Every $G$-chief factor of $Q$ is isomorphic to the irreducible $4$-dimensional module
$V$ of  $X$ described in Lemma~\ref{74 unique}.  Let $N$
be a minimal $G$-invariant subgroup of $Q$. Then $N$ has order $7^4$ and $Q/N$ is elementary abelian. Since every maximal $G$-invariant subgroup is also elementary abelian, we infer  that $Q$ has a unique
such subgroup. By Lemma~\ref{cr49}, $Q/N$ is completely reducible. Therefore $Q/N$ has order $7^4$ and so $Q $ is homocyclic of order
$49^4$ and   Lemma~\ref{notin} provides the contradiction. Thus, if $Q$ is abelian, it is elementary abelian and   completely reducible as a $\GF(7)X$-module.

Suppose now, aiming for a contradiction, that $G$ is chosen so  that $Q$ is a class two group of minimal order.  Thus $Q'$ is elementary abelian of order $7^4$ and $Q/Q'$ is elementary abelian by the previous paragraph. In addition we may assume that $Q'= Z(Q)$ as $Z(Q)$ is completely reducible as a $\GF(7)X$-module. Since the
commutator map from $Q/Z(Q)\times Q/Z(Q)$ determines an $\GF(r)X$-module epimorphism  from $Q/Z(Q)\otimes Q/Z(Q)$ to
$Q'=Z(Q)$, and $V \otimes V$ has no quotients isomorphic to $V$ by Lemma~\ref{74 unique}, we have a contradiction. This proves the lemma.
\end{proof}

\begin{lemma}\label{notin2} The group $\Alt(6)$ is not isomorphic to a subgroup of
$\GL_4(\mathbb Z/9\mathbb Z)$.\end{lemma}

\begin{proof}  Take the integral
copy of $A\cong \Alt(5)$ from  Section~2 and determine the
normaliser $N$ of a Sylow $2$-subgroup using {\sc Magma}. This has order $7776$. None
of the  groups $\langle A,x\rangle$, $x \in N$,  is isomorphic to
$\Alt(6)$. This proves the claim.
\end{proof}

Because of Lemmas~\ref{coho} and \ref{notin2} we can now prove our version of Theorem~\ref{MainThm} (iv) using similar arguments to
those used towards the end of the proof of Lemma~\ref{49}.

\begin{lemma} If $\Alt(6)$ acts on an abelian $3$-group $Q$ with an element of order $5$ acting fixed point freely, then $Q$ is elementary
abelian and completely reducible as a $\GF(3)\Alt(6)$-module.\qed
 \end{lemma}

\end{document}